\documentclass[12pt, a4paper]{amsart}
\usepackage{amssymb}
\usepackage{amsmath}
\usepackage{verbatim}
 \usepackage{color}
 
\newcommand\AND{\quad\text{and}\quad}

\newcommand\circa{\textbf{A}}  
\newcommand\circb{\textbf{B}}  
\newcommand\circc{\textbf{C}}  
\newcommand\circd{\textbf{D}}  
\newcommand\circe{\textbf{F}}  

\newcommand\ep{\varepsilon}

\newcommand\Ex{\mathsf{E}}

\newcommand\K{\mathbb K}

\newcommand\Ll{\mathsf{L}}

\newcommand\Mmin{\underline{M\!}\,}
\newcommand\Mmax{\overline{\!M}}
\newcommand\mf{\mathfrak{m}}
\newcommand\N{\mathbb N}

\newcommand\OQ{(0\,,\,\infty)^2}

\newcommand\Prob{\mathbb{P}}

\newcommand\R{\mathbb R}

\newcommand\scs{\scriptstyle}

\newcommand\utau{T_{\min}}

\newcommand\uno{\mathbf{1}}
 
\newcommand\Var{\operatorname{\sf Var}}

\newcommand\wt{\widetilde}

\newcommand\Z{\mathbb Z}
 
\numberwithin{equation}{section}

\newtheoremstyle{mythm}
  {9pt}
  {9pt}
  {\itshape}
  {0pt}
  {\bfseries}
  {}
  { }
  {\thmnumber{(#2)}\thmname{ #1}\thmnote{ #3}}

\newtheoremstyle{mydef}
  {9pt}
  {9pt}
  {\normalfont}
  {0pt}
  {\bfseries}
  {}
  { }
  {\thmnumber{(#2)}\thmname{ #1}\thmnote{ #3}}

\theoremstyle{mythm}
\newtheorem{thm}[equation]{Theorem.}
\newtheorem{pro}[equation]{Proposition.}
\newtheorem{lem}[equation]{Lemma.}

\theoremstyle{mydef}
\newtheorem{ass}[equation]{Basic Assumption.}
\newtheorem{dfn}[equation]{Definition.}

\begin{document}$\,$ \vspace{-1truecm}
\title{Recurrence of 2-dimensional queueing processes,\\ and
random walk exit times from the quadrant}
\author{\bf Marc PEIGN\'E and Wolfgang WOESS}
\address{\parbox{.8\linewidth}{Laboratoire de Math\'ematiques
et Physique Th\'eorique\\
Universit\'e Francois-Rabelais Tours\\
F\'ed\'eration Denis Poisson -- CNRS\\
Parc de Grandmont, 37200 Tours, France\\}}
\email{peigne@lmpt.univ-tours.fr}
\address{\parbox{.8\linewidth}{Institut f\"ur Diskrete Mathematik,\\ 
Technische Universit\"at Graz,\\
Steyrergasse 30, A-8010 Graz, Austria\\}}
\email{woess@TUGraz.at}
\date{June 15, 2020} 
\thanks{The first author acknowledges support by a visiting professorship
at TU Graz. 
The second author was supported by Austrian Science Fund projects FWF P31889 and W1230
as well as from the European Research Council (ERC) 
under Kilian Raschel's Starting Grant Agreement No759702}

\begin{abstract}
Let $X = (X_1, X_2)$ be a 2-dimensional random variable and $X(n), n \in \N$, a sequence of i.i.d. 
copies of $X$. The associated random walk is $S(n)= X(1) +   \cdots +X(n)$. The corresponding
absorbed-reflected walk $W(n), n \in \N$, in the first quadrant is given by 
$W(0) = x \in \R_+^2$ and
$W(n) = \max \{ 0, W(n-1) - X(n) \}$, where the maximum is taken coordinate-wise.
This is often called the Lindley process and models the waiting times in a two-server
queue. We characterize recurrence of this process, assuming suitable, rather mild
moment conditions on $X$. It turns out that this is directly related with the tail asymptotics
of the exit time of the random walk $x + S(n)$ from the quadrant, so that the main part
of this paper is devoted to an analysis of that exit time in relation with the drift vector, 
i.e., the expectation of $X$.
\end{abstract}

\maketitle

\markboth{{\sf M. Peign\'e and W. Woess}}
{{\sf 2-dimensional queueing process and random walk exit times}}
\baselineskip 15pt

\vspace*{-.5cm}

\section{Introduction}\label{sec:intro}

The waiting times in a single server queue are modeled as a Markov chain $W(n)=W^x(n)$, $n \ge 0$,
on the non-negative half-axis $\R_+\,$. Here, $x \ge 0$ is the initial value, $W(0)=x$, and
$W(n) = \max \{ 0, W(n-1) - X(n) \}$, where the $X(n)$, $n \ge 1$, are i.i.d. real random variables.
We think of this process as an absorbing-reflecting random walk on $\R_+$. Namely, setting 
$S(n)=X(1)+\dots+X(n)$, the process evolves as the random walk $x - S(n)$ as long as it stays 
non-negative. Only when it attempts to cross $0$ and become negative, the new value is
reset to $0$ before continuing. This is often called the \emph{Lindley process,}
see {\sc Lindley}~\cite{Li}. There is an extensive literature on this Markov chain, such as
the seminal paper by {\sc Kendall}~\cite{Ke} plus the references therein, 
and the monographs by {\sc Feller}~\cite{Fe2},
{\sc Borovkov}~\cite{Bo} and {\sc Asmussen}~\cite{As}.

In the present work, we are interested in the multi-dimensional case, and more precisely, the 2-dimensional
one, where the $X(n) = \bigl( X_1(n), X_2(n)\bigr)$ are i.i.d. copies of a 2-dimensional 
random variable $X = (X_1\,,X_2)$, the starting point $x$ lies in the non-negative quadrant
of $\R^2$, and the maximum in $W(n) = W^x(n) = \max \{ 0, W(n-1) - X(n) \}$ is taken coordinate-wise.
There are various types of
multi-dimensional queueing processes. A first rigorous study is due to 
{\sc Kiefer and Wolfowitz}~\cite{KiWo}. In that reference, there are several servers and only one queue, while our process corresponds to having as many queues as there are servers. {\sc Kingman}~\cite{Ki},
{\sc Flatto and McKean}~\cite{FM} and others studied the model with two servers and two queues,
were arriving customers choose the shorter queue. {\sc Greenwood and Shaked}~\cite{GrSh} touch the 
model where the process restarts from the origin whenever just one of the coordinates
becomes negative. Again, this is different from our situation, for which there are not many
references.
Regarding recurrence of the 2-dimensional Lindley process, see e.g. the recent note of {\sc Cygan and Kloas}~\cite{CyKl}. 

Another viewpoint is to consider the 2-dimensional absorbing-reflecting random walk $W^x(n)$ as a 
stochastic dynamical system (SDS) evolving in $\R_+^2\,$, as in 
the work of {\sc Leguesdron}~\cite{Le}, {\sc Peign\'e} \cite{Pe}, {\sc Benda}~\cite{Be}, \cite{Be1},
\cite{Be2}.
The viewpoint of the last references 
-- inspired by important work of {\sc Babillot, Bougerol and Elie}~\cite{BBE} -- 
has been to exploit the fact that 
this SDS is obtained by iterating a sequence of i.i.d. random contractions of
$\R_+^2$, and that it is
\emph{locally contractive.} See {\sc Peign\'e and Woess}~\cite{PeWo1}, \cite{PeWo2}
for the precise definition, an outline of relevant parts of Benda's (not easily accessible) work, and
further results concerning conservativity, ergodicity, and invariant measures.

Indeed, when the one-dimensional marginal processes are recurrent -- a well-understood
situation, and the one we are interested in here -- then our SDS is \emph{strongly contractive,}
that is 
$$
|W_n^x - W_n^y| \to 0 \quad \text{almost surely, for all }\; x,y \in \OQ\,.
$$
This implies that the process is either \emph{transient,} that is, $W^x(n) \to \infty$
almost surely, or else it is recurrent in the following sense: there is a closed, non-empty 
\emph{limit set}  $\Ll \subset \R_+^2$ such that for any open set $U$ which intersects $\Ll$
and any starting point $x$, we have 
\begin{equation}\label{eq:limitset}
\Prob[ W^x(n) \in U \; \text{for infinitely many}\; n ] =1.
\end{equation} 
A main goal of this paper is to understand when the 2-dimensional Lindley process
is recurrent. 

\begin{ass}\label{ass:basic}
Throughout this paper, we assume that
\begin{itemize}
 \item[(i)] $\;X$ is not constrained to a hyperplane, and that
\item[(ii)] $\;\Prob[X \in \OQ] > 0\,.$
\end{itemize}
 \end{ass}

This is a quite natural assumption on the support of the distribution of $X = (X_1\,,X_2)$; 
see the discussion in \S \ref{sec:Lindley}.
Based on this assumption and the preceding observations, the following is our main result concerning 
the two-dimensional Lindley process, with quite general moment assumptions for the
random variable $X$ which models the increments, and its positive and
negative parts $X_i^+ = \max\{ X_i\,,0\}$ and $X_i^- = \max\{ -X_i\,,0\}$, $i=1,2$.

\begin{thm}\label{thm:main1} \emph{(a)} If for at least one $i \in\{1,2\}$,
$$
\Ex(X_i^+) < \Ex(X_i^-) \le \infty\,,
$$
then $W(n)$ is transient.
\\[4pt]
\emph{(b)} If for both $i \in\{1,2\}$,
$$
\Ex(X_i^-) < \Ex(X_i^+) < \infty\,,
$$
then $W(n)$ is positive recurrent.
\\[4pt]
\emph{(c)} If for both $i \in\{1,2\}$,
$$
\Ex(X_i)=0 \AND 
\Ex\bigl(|X_i|^{\max\{ 2+\delta\,,\, \pi/\arccos(-\rho)\}}\bigr) <  \infty\,,
$$
where $\delta > 0$ (arbitrary) and $\rho = \rho(X_1\,,X_2)$ is the correlation coefficient of
$X_1$ and $X_2\,$, then $W(n)$ is null recurrent if $\rho \ge 0$, and transient,
otherwise.
\\[4pt]
\emph{(d)} If for some $\delta > 0$, up to a possible exchange of the two coordinates,
$$
\Ex(|X_i|^{2+\delta}) <  \infty \quad \text{for }\;i =1,2, \quad
\Ex(X_1)=0, \quad \Ex(X_2) > 0, \AND 
\Ex((X_2^-)^{3+\delta}) < \infty\,,
$$
then $W(n)$ is null recurrent.
\end{thm}

\medskip

The two interesting cases are (c) and (d). In order to prove the theorem, we use the 
following.

\begin{dfn}\label{def:exit}
For $x \in \OQ\,$, the exit time of the random walk from the positive 
quadrant is
$$
\tau_x = \inf \{ n\ge 1 : x + S(n) \notin \OQ\}\,.
$$ 
\end{dfn}

In cases (c) and (d), $\tau_x$ is almost surely finite.
We have the following relation of the Lindley process with these exit times. 

\begin{lem}\label{lem:relation} For $x = (x_1\,,x_2) \in \OQ\,$,
$$
\Prob\bigl [ W^0(n) \in [0,x_1) \times [0,x_2) \bigr] = \Prob[ \tau_x > n].
$$
\end{lem}

Thus, we are led to another topic of its own big interest, currently the object
of very active work: the tail asymptotics of exit times of random walks from 
cones, in our case, the positive quadrant. Our main source is the profound paper of
{\sc Denisov and Wachtel}~\cite{DeWa} on exit times from cones for centered 
random walks. As we shall see, this leads to statement (c) of Theorem
\ref{thm:main1}. Our main focus is on the case when one coordinate is centered and
the second one has positive drift. Exit times from cones for random walks with
drift were considered by {\sc Duraj}~\cite{Du} and {\sc Garbit and Raschel}~\cite{GaRa}, 
who however do not provide the tail asymptotics which we need. The central body of the
present paper concerns that case. We summarize.

\begin{thm}\label{thm:main2} Under the same moment conditions as in
Theorem \ref{thm:main1}, we have the following for any 
$x \in \OQ\,$, as $n \to \infty$.
\\[4pt]
In case \emph{(a)}, \quad if $\;\Ex(X_i) < 0\;$ and $r\ge 1$ 
then $\;\Ex\bigl((X_i^+)^r\bigr) < \infty\;$
implies $\;\Ex(\tau_x^r) < \infty\,$, whence  $\Prob[\tau_x > n] = \mathfrak{o}(n^{-r})$.
\\[4pt]
In case \emph{(b)}, $\quad \lim\limits_{n \to \infty} \Prob[\tau_x > n] = \Prob[\tau_x = \infty] > 0$
\\[4pt]
In case \emph{(c)}, $\quad \Prob[\tau_x > n] \sim v(x)\, n^{-1/p}\,,\quad 
\text{where} \quad p =2\arccos (-\rho)/\pi\,.$
\\[4pt]
In case \emph{(d)}, $\quad \Prob[\tau_x > n] \sim \kappa \,h(x)\, n^{-1/2}\,.$
\\[4pt] 
In (c) and (d), $v$ and $h$ are positive harmonic functions for the respective 
random walk $x + S(n)$
killed when exiting $\OQ\,$, and $\kappa > 0$.
\end{thm}

(The reader will see below why we prefer not to incorporate the constant 
$\kappa$ into $h(x)$.)

\medskip

These are the main results of the present paper. The main body of the work consists
in the proof of Theorem \ref{thm:main2}(d), preceded by (a)--(c). 

\smallskip

In \S \ref{sec:dim1}, we collect some preliminary facts concerning the one-dimensional centered
case. Furthermore, the
asymptotic behaviour of the corresponding harmonic function on $\R^+$ is crucial for us.
V. Wachtel has comunicated to us an improved version, due to himself and D. Denisov, 
of the outline in \cite[2.4]{DeWa}.
With their kind permission , this is explained in the Appendix of this paper.

In \S \ref{sec:dim2}, we provide the proofs of Theorem \ref{thm:main1}.
Only in \S \ref{sec:Lindley}, we come back to the queueing process, i.e., 
the proofs of Lemma \ref{lem:relation} 
and Theorem \ref{thm:main1} plus additional observations and a discussion.

\section{Dimension one: exit times from the half-line}\label{sec:dim1}
 
In this section only, $X$ is a one-dimensional random variable, $X(n)$, $n \in \N$, 
are i.i.d. copies of $X$, and $S(n)=X(1)+\dots+X(n)$ (with $S(0)=0$).
For $x \in (0\,,\,\infty)$,
\begin{equation}\label{eq:taux1}
\tau_x = \inf \{ n \in \N : x + S(n) \le 0 \}.
\end{equation}
We collect several facts concerning $\tau_x$ and related random variables, such as
\begin{equation}\label{eq:minmax}
\Mmin(n) = \min \{ S(k) : 0 \le k \le n\} \AND \Mmax(n) = \max \{ S(k) : 0 \le k \le n\}.
\end{equation}

For the following, see the monograph by {\sc Gut}~\cite[Thm. 3.1]{Gu}.

\begin{lem}\label{lem:negdrift}
If $X$ is integrable and $\Ex(X) < 0$ then $\Ex(\tau_x) < \infty\,$. In addition, for any $r \ge 1$,
$$
\Ex\bigl((X^+)^r\bigr) < \infty \iff \Ex(\tau_x^r) < \infty\,.
$$
\end{lem}

For the following, \cite{Gu} is again a good source, as well as the very nice 
exposition by {\sc Janson}~\cite{Ja}.

\begin{lem}\label{lem:posdrift}
 If $X$ is integrable and $\Ex(X) > 0$ then 
$$
\Prob[\tau_x = \infty] > 0 \AND  \Mmin = \inf \{ S(n): n \ge 0\} > -\infty \quad 
\text{almost surely,}
$$
so that $\;\utau = \inf \{ n \ge 0: S(n) = \Mmin \} < \infty\;$ almost surely.
Furthermore, for any $r > 0$,
$$
\Ex\bigl((X^-)^{r+1}\bigr) < \infty \iff \Ex\bigl(|\Mmin|^r \bigr) 
< \infty \iff \Ex(\tau_x^r\,;\, \tau_x < \infty) < \infty
\iff \Ex\bigl(\utau^r \bigr) < \infty\,.
$$
\end{lem}

For the remainder of this section, we assume that $X$ is centered. 

\begin{pro}\label{pro:harmonic1}
Suppose that $\Ex(X^2)<\infty$ and $\Ex(X)=0$. Then the
function 
$$
h_1(x) = x - g_1(x)\,,\quad \text{where}\quad g_1(x) = 
\Ex\bigl(x+S(\tau_x)\,;\,\tau_x < \infty\bigr)\,,\quad x > 0, 
$$
is finite and harmonic for the random walk killed when exiting $(0\,,\,\infty)$,
that is,
$$
h_1(x) = \Ex\Bigl( h_1\bigl(x+S(1)\bigr)\,;\, \tau_x > 1\Bigr). 
$$  
There is $d > 0$ such that
\begin{equation}\label{eq:h1}
h_1(x) \ge \max\{x,d\} \AND
\lim_{x \to \infty} \frac{h_1(x)}{x} =1.
\end{equation}
\end{pro}

\begin{proof}
As $\Ex(X^2) < \infty\,$, finiteness of 
$g_1(x)$ follows from classical results, see e.g. \cite{Fe2}. Note that $g_1(x) \le 0$,
whence $h_1(x_1) \ge x_1\,$. Since $X$ is non-degenerate and centered, there is $c > 0$ such
that $\Prob[X \le - 2c] \ge 1/3$. Thus, if $0 < x \le c$ then $\Prob[x+S(1) \le c] \ge 1/3$,
whence $-g_1(x) \ge 1/3$. This proves the first part of \eqref{eq:h1}.

Harmonicity is easily proved: Since $\Ex(X)=0$,
$$
\begin{aligned}
x & = \int_{y > -x} (x+y) \, \Prob[X=dy] + \int_{y \le -x} (x+y) \, \Prob[X=dy] \\
&= \Ex\bigl( x+S(1)\,;\, \tau_x > 1\bigr) + \int_{y \le -x} (x+y) \, \Prob[X=dy] .
\end{aligned}
$$
On the other hand, decomposing with respect to the first step of the walk,
$$
\begin{aligned}
g_1(x) &= 
\int_{y > -x} \Ex\bigl(x+y+S(\tau_{x+y})\,;\,\tau_{x+y} < \infty\bigr)\, \Prob[X=dy]
+ \int_{y \le -x} (x+y) \, \Prob[X =dy]\\
&= \Ex\bigl( g_1\bigl(x+S(1)\bigr)\,;\, \tau_x > 1\bigr)
 + \int_{y \le -x} (x+y) \, \Prob[X =dy]\,.\\
\end{aligned}
$$
Taking the difference, we get harmonicity of $h_1$ as required.

\smallskip

Finally, 
a proof by D. Denisov and V. Wachtel that $h_1(x)/x \to 1$ as $x \to \infty$ 
was comunicated to us by V. Wachtel. It is provided in the Appendix. 
\end{proof}

In the next section, we shall need an additional estimate related to the function $g_1(x)$. 
Let $(\ell_k)_{k \ge 0}$ be the strictly decreasing ladder indices of $\bigl(S(n)\bigr)_{n \ge 0}\,$. 
That is, $\ell_0=0$, and $\ell_k = \min \{n: S(n) < S(\ell_{k-1})\}$.
Then the increments $S(\ell_n) - S(\ell_{n-1})$  are i.i.d. and integrable, because 
$\Ex(X^2) < \infty$. Also, $\tau_x = S(\ell_k)$ for some random $k$ depending on $x > 0$.
\begin{lem}\label{lem:eqi}
For each $m > 0$ there is a constant $U(m) > 0$ such that
$$
\Prob[x + S_1(\tau_x) < - t] \le U(m)\,\Prob[S({\ell_1}) < -t] \quad \text{for all }\; 
t > 0\,,\; x \in (0,m]\,.
$$
\end{lem}
\begin{proof} 
For $x > 0$, let 
$$
U(x) = \sum_{k=0}^{\infty} \Prob\bigl[S(\ell_k) \in (-x\,,\,0]\,\bigr]\,. 
$$
This is the potential kernel of $\bigl(S(\ell_k)\bigr)_{k \ge 0}\,$. It is finite, since 
$S(\ell_k)$ is a sum of i.i.d. random variables tending almost surely to $-\infty\,$. Then
$$
\begin{aligned}
 \Prob[x + S_1(\tau_x) < - t] &= \sum_{k=1}^{\infty} \Prob\bigl[ \tau_x 
= \ell_k\,,\; x+ S(\ell_k)< - t\bigr]\\
 &\le \sum_{k=1}^{\infty} \Prob\bigl[ x + S(\ell_{k-1}) > 0\,,\; S(\ell_k)-S(\ell_{k-1}) < - t\bigr]\\
 &=\sum_{k=1}^{\infty} \Prob\bigl[ S(\ell_{k-1}) \in (-x\,,\,0]\,\bigr] \times \Prob[S(\ell_1) < - t\bigr]
 = U(x) \,\Prob[S(\ell_1) < - t\bigr]\,.
\end{aligned}
$$
Since $U(x) \le U(m)$ for $x \le m$, this concludes the proof.
\end{proof}

We shall also need the following Lemma.
The first part is in principle known; the second part is adapted from 
{\sc Pham}~\cite[Lemma 4.5]{Ph}.
For the sake of completeness we provide a proof.

\begin{lem}\label{lem:Mn}
Suppose that $\Ex(X^2) < \infty$ and $\Ex(X)=0$. Then for any $p>2$ there is a 
constant $c_p > 0$ such that when \/ $\Ex(|X|^p) < \infty\,$, for each $t > 0$ and $n \in \N$ 
$$
\Ex\bigl(|S(n)|^p\bigr) \le c_p\, n^{p/2}\,\Ex\bigl(|X|^p\bigr) \AND
\Prob[\,\Mmax(n) > t] \le c_p\,\Ex\bigl(|X|^p\bigr)\, n^{p/2}\, t^{-p}\,.
$$
Furthermore, for any $\alpha > 0$, 
$$
\Ex \bigl( n^{1/2}+\Mmax(n)\,;\, \Mmax(n) > n^{1/2 + \alpha}\bigr)
\le c_p \, \Ex\bigl(|X|^p\bigr)\,\frac{2p-1}{p-1}\, n^{1/2 - (p-1)\alpha}\,.
$$
\end{lem}

\begin{proof}
Since $\bigl( S(n) \bigr)$ is a martingale, Doob's maximal inequality implies
$$
\Prob[\,\Mmax(n) > t] \le \Ex\bigl(|S(n)|^p\bigr)\, t^{-p}\,.
$$ 
By a well-known inequality which one finds e.g. in 
{\sc Burkholder}~\cite[Proof of Thm. 3.2]{Bu},
$$
\Ex\bigl(|S(n)|^p\bigr) \le c_p\, \Ex\biggl( \Bigl( \sum_{k=1}^n X(k)^2 \Bigr)^{p/2}\biggr).
$$
Now H\"older's inequality concludes the proof of the first two inequalities.
For the third inequality,
 $$
\begin{aligned}
\Ex \bigl( n^{1/2}+\Mmax(n)\,&;\, 
\Mmax(n) > n^{1/2 + \alpha}\bigr) \\
&= \bigl( n^{1/2} + n^{1/2+\alpha} \bigr) \,\Prob\bigl[\,\Mmax(n) > n^{1/2 + \alpha}\bigr]
+ \Ex\Bigl( \bigl(\, \Mmax(n) - n^{1/2 + \alpha} \bigr)^+\Bigr)\\
&\le 2 n^{1/2+\alpha} \,\Prob\bigl[\,\Mmax(n) > n^{1/2 + \alpha}\bigr] 
+ \int_{n^{1/2+\alpha}}^{\infty} \Prob\bigl[\,\Mmax(n) > t\bigr] \, dt \\
&\le 2 \, c_p \, \Ex\bigl(|X|^p\bigr)\, n^{1/2 - (p-1)\alpha}
+ c_p \, \Ex\bigl(|X|^p\bigr)\int_{n^{1/2+\alpha}}^{\infty} n^{p/2}\, t^{-p}\, dt\\
&= c_p \, \Ex\bigl(|X|^p\bigr)\,\frac{2p-1}{p-1}\, n^{1/2 - (p-1)\alpha}\,,
\end{aligned}
$$
as proposed.
\end{proof}
 
Uniform convergence in the following proposition was proved
by {\sc Doney}~\cite{Do}. 

\begin{pro}\label{pro:dim1}
If  $\Ex\bigl(|X|^2\bigr) < \infty$ and $\Ex(X)=0$, then 
$$
\Prob[\tau_x > n] \sim \kappa\, h_1(x) \,n^{-1/2}\,,\quad \text{where}
\quad \kappa = \bigl( \pi\, \Var(X)/2  \bigr)^{-1/2}\,,
$$
uniformly as $n \to \infty$ and $0 < x < \theta_n \,n^{1/2}$, where $(\theta_n)$
is an arbitrary positive sequence that converges to $0$.  
\end{pro}

\section{Dimension two: exit times from the quadrant}\label{sec:dim2}

This section is devoted to the proof of Theorem \ref{thm:main2}. 
We return to the situation of the Introduction, where $X = (X_1\,,X_2)$ and
the $X(n)$ are i.i.d. copies of $X$. Thus, $S(n) = \bigl(S_1(n),S_2(n)\bigr)$, and
if $x = (x_1\,,x_2) \in \OQ$ then
\begin{equation}\label{eq:taumin}
\tau_x = \min \{ \tau_{x_1}\,,\tau_{x_2}\}\,,\quad \text{where}\quad
\tau_{x_i} = \tau_{i,x_i} = \inf \{ n \in \N : x_i + S_i(n) \le 0\,,\; i=1,2\}\,. 
\end{equation}

\goodbreak

\textbf{Negative drift in at least one coordinate}

\smallskip

\begin{proof}[Proof of Theorem \ref{thm:main2}(a)]
It is well known that if $\Ex(X_i) < 0$ then $\;\Ex\bigl((X_i^+)^r\bigr) < \infty\;$
if and only if $\;\Ex(\tau_{x_i}^r) < \infty\,$, see e.g. \cite{Fe2}, \cite{Gu} and  
\cite{Ja}. In view of \eqref{eq:taumin}, this implies $\Ex(\tau_{x}^r) < \infty\,$.
\end{proof}

\bigskip\goodbreak

\textbf{Positive drift in both coordinates}

\smallskip

\begin{proof}[Proof of Theorem \ref{thm:main2}(b)] We use Lemma \ref{lem:posdrift}
 and set $\Mmin_i = \inf\{ S_i(n): n \ge 0\}$. These random variables are almost
surely finite. Therefore, for every $n$,
$$
\Prob[\tau_x > n] = \Prob[\tau_{x_1} > n\,,\, \tau_{x_2} > n]
\ge \Prob[ \Mmin_1 < -x_1\,,\, \Mmin_2 < - x_2] \to 1 \,,\quad\text{as }\;
x_1\,,x_2 \to \infty.
$$
Thus, there is $b > 0$ such that $\Prob[\tau_x = \infty] \ge 1/2$ if 
$x_1\,,x_2 > b$.  If $x \in \OQ$ is arbitrary then the Basic Assumption \ref{ass:basic},
namely that $\Prob[X \in \OQ] > 0$ yields that with positive probability, the 
random walk $x + S(n)$ starting at $x$ can reach $(b\,,\,\infty) \times (b\,,\,\infty)$
without exiting $\OQ$. Therefore $\Prob[\tau_x = \infty] > 0$. 
\end{proof}

\bigskip

\textbf{The centered case}

\smallskip

\begin{proof}[Proof of Theorem \ref{thm:main2}(c)] We assume that both coordinates
of $X$ have finite moment of order $\max\{ 2+\delta\,,\, \pi/\arccos(-\rho)\}$
and are centered, with some $\delta > 0$ and $\rho=\rho(X_1\,,X_2) \in (-1\,,\,1)$. 

The result is in principle stated in \cite[Example 2]{DeWa}. For the sake of
completeness, we give a few hints. 
Let $\K_{\alpha}$ be a standard closed cone in $\R^2$: the sides
of the cone are two half-lines issuing from the origin, which is the cone's vertex. The opening 
angle of the cone is $\alpha \in (0\,,\,\pi)$. On $\K_{\alpha}\,$, the paper \cite{DeWa} considers
a random walk $\tilde x+\wt S(n)$ with $\tilde x \in \K_{\alpha}^o$ and 
$\wt S(n) = \wt X(1) + \dots+ \wt X(n)$, where the $\wt X(n)$ are i.i.d. copies
of $\wt X = (\wt X_1\,,\wt X_2)$. The assumptions are that 
the coordinates of $\wt X$ have finite moment of order 
$\max\{ 2+\delta\,,\, \pi/\alpha\}$, are centered, with $\Ex(\wt X_i^2)=1$ and 
$\rho(\wt X_1\,,\wt X_2) = \Ex(\wt X_1 \wt X_2)=0$. For the associated exit time 
$\wt \tau_{\tilde x}\,$, the main results of \cite{DeWa} yield that
$$
\Prob[\wt \tau_{\tilde x} > n] \sim V(\tilde x) \, n^{-1/p}\,,\quad \text{where}
\quad p = 2\alpha/\pi\,.
$$
Thus, the method is to decorrelate $X_1$ and $X_2$ and thereby to pass from $\R_+^2$ to
a possibly modified cone. This can for example be achieved by the matrix transformation
$$
\begin{pmatrix}\wt X_1\\ \wt X_2 \end{pmatrix} = 
M\, \begin{pmatrix}X_1\\ \wt X_2 \end{pmatrix}\,,\quad \text{where} \quad
M = \begin{pmatrix} \frac{1}{\sigma_1} & 0 \\[4pt] \frac{-\rho}{\sigma_1\sqrt{1-\rho^2}} & 
\frac{1}{\sigma_2}\end{pmatrix} 
$$
Here, $\sigma_i^2 = \Ex(X_i^2)$.
Then $(\wt X_1\,,\wt X_2)$ are centered, with variances $=1$, and non correlated.
The mapping $x \mapsto Mx$ transforms $\R_+^2$ in a cone $\K_{\alpha}$ with 
$\alpha = \arccos(-\rho)$.  Thus, we have
$$
\Prob[\tau_x > n] = \Prob[\wt \tau_{Mx} > n] \sim V(Mx) \, n^{-1/p}\,,\quad \text{where}
\quad p = 2\arccos(-\rho)/\pi\,.
$$
Since $V(\tilde x)$ is positive harmonic for the random walk killed when exiting $\K_{\alpha}^o$,
the function $v(x) = V(Mx)$ is positive harmonic for the original random walk killed when exiting
$\OQ$.
\end{proof}

\bigskip

\textbf{The mixed case positive drift - zero drift}

\smallskip

We finally come to the main body of this paper, namely the proof of statement (d) of 
Theorem \ref{thm:main2}. We repeat that we assume that $X_1$ is centered with moment
of order $2+\delta$ and that $X_2$ has finite second moment, $\mf= \Ex(X_2) > 0$ and
in addition $\Ex\bigl((X_2^-)^{3+\delta}\bigr) < \infty$ for some $\delta > 0$.

The first step consists in finding the right harmonic function
for the random walk restricted to the quadrant. For this, we were inspired by
the work of {\sc Ignatiouk-Robert and Loree}~\cite{IgLo}, who however assume exponential
moment conditions.

\begin{pro}\label{pro:harmonic}
Under the assumptions (d), the function
$$
h(x) = x_1 - \Ex\bigl( x_1+ S_1(\tau_x)\,;\, \tau_x < \infty \bigr)\,,
\quad x \in \OQ\,,
$$
is strictly positive and harmonic for the random walk killed upon exiting the positive quadrant.
It satisfies
\begin{equation}\label{eq:h}
h(x) \le h_1(x_1) \AND \lim_{x_2 \to \infty} \frac{h(x)}{h_1(x_1)} = 1 \quad \text{uniformly for }\;
0 < x_1\le x_2^{2+\delta}\,,
\end{equation}
where $h_1$ is the function of Proposition \ref{pro:harmonic1}.
\end{pro}

\begin{proof} 
The proof that $h$ is harmonic is a straightforward adaptation of the proof
in the one-dimensional case, Proposition \ref{pro:harmonic1}.

\smallskip

Write $g(x) = \Ex\bigl( x_1+ S_1(\tau_x)\,;\, \tau_x < \infty \bigr)$. Then by \eqref{eq:taumin},
$$
\begin{aligned}
g(x) &= \tilde g(x) + f(x) = g_1(x_1) - \tilde f(x) + f(x)\,, \quad\text{where}\\
\tilde g(x) &= \Ex\bigl( x_1+ S_1(\tau_{x_1})\,;\, \tau_{x_1}\le \tau_{x_2}\,,\; \tau_{x_1} < \infty \bigr)\,,\\
f(x) &= \Ex\bigl( x_1+ S_1(\tau_{x_2})\,;\, \tau_{x_2} < \tau_{x_1} < \infty \bigr)\,,\AND\\
\tilde f(x) &= \Ex\bigl( x_1+ S_1(\tau_{x_1})\,;\, \tau_{x_2} < \tau_{x_1} < \infty \bigr)\,,
\end{aligned}
$$
and $g_1$ is the function from Proposition \ref{pro:harmonic1}. 
Since $\tilde f(x) \le 0$ and $f(x) \ge 0$, we see that $g(x)\ge g_1(x_1)$, whence
$h(x) \le h_1(x_1)$. 
\\[4pt]
\underline{Claim 1.} \hspace*{2.3cm} $\lim\limits_{x_2\to \infty} f(x) = 0 \quad \text{uniformly for }\;
0 < x_1\le x_2^{2+\delta}\,.$

\smallskip

To prove this, we use Lemma \ref{lem:posdrift}: 
since $\Ex\bigl((X_2^-)^{3+\delta}\bigr) < \infty\,$,
$$
\Ex(\tau_{x_2}^{2+\delta}\,;\, \tau_{x_2} < \infty) < \infty \AND
\Ex(\Mmin_2^{2+\delta}\,;\, \tau_{x_2} < \infty) < \infty \,, 
$$
where $\Mmin_2 = \inf \{ S_2(n): n \ge 0\}\,$.
Write $\sigma_1^2 = \Ex(X_1^2)$, and recall that $\tau_{x_1} < \infty$ almost surely. 
For every $k \in \N$, using the Cauchy-Schwarz inequality,
$$
\Ex\bigl( x_1+ S_1(k)\,;\, \tau_{x_1} > \tau_{x_2} = k \bigr) \le
x_1\,\Prob[\tau_{x_1} > \tau_{x_2} = k] + 
\underbrace{\Ex\bigl(S_1(k)^2\bigr)^{1/2}}_{\displaystyle \sigma_1 \, k^{1/2}}\, 
\Prob[\tau_{x_2} = k]^{1/2}
$$
We again use the Cauchy-Schwarz inequality in the following estimate.
\begin{equation}\label{eq:CSest}
\begin{aligned}
 &f(x) = \sum_{k=1}^{\infty} \Ex\bigl( x_1+ S_1(k)\,;\, \tau_{x_1} > \tau_{x_2} = k \bigr) \\
&\le x_1\,\Prob[\tau_{x_1} > \tau_{x_2}] +  
\sigma_1 \sum_{k=1}^{\infty} k^{-(1+\delta)/2} \cdot k^{(2+\delta)/2} \, \Prob[\tau_{x_2} = k]^{1/2}\\
&\le x_1\,\Prob[\tau_{x_1} > \tau_{x_2}] + C_{\delta}\,
\Ex(\tau_{x_2}^{2+\delta}\,;\, \tau_{x_2} < \infty)^{1/2}\,, \quad\text{where}\quad
C_{\delta}= \sigma_1 \,\left(\sum_{k=1}^{\infty} k^{-(1+\delta)}\right)^{\!1/2}.
\end{aligned}
\end{equation}
Now, our moment assumption implies that 
\begin{equation}\label{eq:Min}
\Prob[\tau_{x_2} < \infty] \le  \Prob[\Mmin_2 \le -x_2] = \mathfrak{o}(x_2^{-(2+\delta)}) 
\quad \text{as }\; x_2 \to \infty\,.
\end{equation}
Therefore 
$$
\lim_{x_2 \to \infty} x_1\,\Prob[\tau_{x_1} > \tau_{x_2}] = 0 \quad \text{uniformly for }\;
0 < x_1 \le x_2^{2+\delta}\,.
$$
For the second term, we use Lemma \ref{lem:posdrift}, with $\utau$ referring to 
the second coordinate:
$$
\tau_{x_2}^{2+\delta}\,\uno_{[\tau_{x_2} < \infty]} \le 
\utau^{2+\delta}\,\uno_{[\tau_{x_2} < \infty]}\;
\le \utau^{2+\delta}\,, 
$$
and the middle term tends to $0$ almost surely, as $x_2 \to \infty\,$. By dominated convergence,
\begin{equation}\label{eq:asympbehaviorexpectation}
\lim_{x_2 \to \infty} \Ex(\tau_{x_2}^{2+\delta}\,;\, \tau_{x_2} < \infty) = 0.
\end{equation}
This concludes the proof of Claim 1. 
\\[5pt]
\underline{Claim 2.}  \hspace*{2.2cm}  
$\lim\limits_{x_2\to \infty} \dfrac{\tilde f(x)}{h_1(x_1)} = 0 \quad \text{uniformly for }\;
x_1 \in (0\,,\,\infty)\,.$

\smallskip

To verify this, we first note that since $0 \le -\tilde g(x) \le -g_1(x_1)$,
 we get from Proposition \ref{pro:harmonic1}, resp. the Appendix, that 
$$
\lim_{x_1\to \infty}  \frac{\tilde g(x)}{h_1(x_1)} 
= \lim_{x_1\to \infty}  \frac{\tilde f(x)}{h_1(x_1)} = 0 
\quad\text{uniformly in }\; x_2\,.
$$
Second, Lemma \ref{lem:eqi} yields equintegrability of 
$R_{x_1} = \bigl(-x_1 - S_1(\tau_{x_1})\bigr)\,\uno_{[\tau_{x_1} < \infty]}$ for $x_1 \in [0\,,\,m]$,
for each $m > 0$. Note that $R_{x_1} \ge 0\,$. Since $\Prob[\tau_{x_2} < \infty] \to 0$ as
$x_2 \to \infty$ by \eqref{eq:Min}, a standard argument yields that for each $m > 0$,
$$
\lim_{x_2\to \infty} \tilde f(x) = 0 \quad\text{uniformly for } x_1 \in [0\,,\,m].
$$
Indeed, for $x_1 \in [0\,,\,m]$ and any $t > 0$, 
$$
\begin{aligned}
-\tilde f(x) &\le \Ex(R_{x_1}\,;\, \tau_{x_2} < \infty)\\   
&\le t\, \Prob[R_{x_1} \le t\,,\; \tau_{x_2} < \infty] + \Ex(R_{x_1}\,;\, R_{x_1} > t)\\
&\le t\, \Prob[\tau_{x_2} < \infty] + t\, \Prob[R_{x_1} > t] + 
\sum_{n \ge t} \Prob[R_{x_1} > t + n] \\
&\le t\, \Prob[\tau_{x_2} < \infty] + U(m)\, t\, \Prob[S(\ell_1) <- t] + 
U(m) \sum_{n \ge t} \Prob[S(\ell_1) <-(t + n)]\,.
\end{aligned}
$$
Since $S(\ell_1)$ is negative and integrable, the second and third term in the last
line tend to $0$ as $t \to \infty$. Thereafter, we can chose $x_2$ large enough
so that also the first term is small. Recalling from \eqref{eq:h1}
that $h_1$ is bounded away from zero, we also get that $\tilde f(x)/h_1(x_1) \to 0$
uniformly for $x_1 \in [0\,,\,m]$ as $x_2 \to \infty\,$.

\smallskip

Combining the two parts, we see that Claim 2 is true.

Again using that $h_1$ is bounded away from 0, Claims 1 and 2 prove \eqref{eq:h}.
\\[5pt]
\underline{Claim 3.}  \hspace*{2.5cm} 
$h(x) \ge 0 \quad \text{for all }\; x \in  \OQ\,$.

\smallskip

To see this, recall that $0 \le -\tilde g(x) \le -g_1(x_1) < \infty$, while we know from 
Claim 1 that $0\le f(x) < \infty$. Therefore 
$$
\Ex\bigl( |x_1+ S_1(\tau_x)|\,;\, \tau_x < \infty \bigr) < \infty\,,
$$
and by dominated convergence and the martingale stopping lemma,
$$
\begin{aligned}
g(x) &= \lim_{n \to \infty} g_n(x)\,, \qquad \text{where}\\
g_n(x) &= \Ex\bigl( x_1+ S_1(\tau_x)\,;\, \tau_x \le n \bigr)\\
&= \underbrace{\Ex\bigl( x_1+ S_1(\min\{n, \tau_x\})\bigr)}_{\displaystyle = x_1} - 
\Ex\bigl( \underbrace{x_1+ S_1(n)}_{\displaystyle > 0} \,;\, \tau_x > n \bigr) \le x_1\,.
\end{aligned}
$$ 
This proves Claim 3. 

\smallskip

It remains to prove that $h > 0$ strictly on $\OQ$.
Since $\tilde g(x) \le 0$, the estimate \eqref{eq:CSest} plus the first 
inequality in \eqref{eq:Min} lead to
$$
h(x) \ge x_1 - f(x) \ge  x_1\, \Prob[\,\Mmin_2 > -x_2] - C_{\delta}\,
\Ex(\tau_{x_2}^{2+\delta}\,;\, \tau_{x_2} < \infty)^{1/2}.
$$
By \eqref{eq:asympbehaviorexpectation}, this yields that for every $b_1 > 0$ (possibly small) there is 
$b_2 > 0$ (possibly large) such that $h > 0$ on
$(b_1\,,\,\infty) \times (b_2\,,\,\infty)$. Now, by the Basic Assumption
\ref{ass:basic}, there is $(a_1\,,a_2) \in \OQ$ such that
$\Prob[X \in (a_1\,,\,\infty) \times (a_2\,,\,\infty)] > 0$.
Inductively, we get 
$$
\Prob[S(k) \in (k\, a_1\,,\,\infty) \times (k\,a_2\,,\,\infty) \;\text{for}\; k=1,\dots, n] > 0
$$
for every $n \in \N$. Now let $x \in \OQ$. Given $x_1\,$, by the above there is $b_2$
such that $h > 0$ on $(x_1\,,\,\infty) \times (b_2\,,\,\infty)$. Then there is $n$
such that $x_2 + n\,a_2 \ge b_2$. But then
$$
\Prob[ x+S(n) \in (x_1\,,\,\infty) \times (b_2\,,\,\infty)\,, \tau_x > n] > 0\,.
$$
Therefore, since we already know that $h \ge 0$, 
$$
\begin{aligned}
h(x) &= \Ex\Bigl(h\bigl(x+S(n)\bigr)\,;\, \tau_x > n\bigr)\Bigr)\\
&\ge \Ex\Bigl(h\bigl(x+S(n)\bigr)\,;
\,x+S(n) \in (x_1\,,\,\infty) \times (b_2\,,\,\infty)\,, \tau_x > n\bigr)\Bigr)
 > 0\,. 
\end{aligned}
$$
This completes the proof of the proposition.
\end{proof}

We now choose $\ep \in (0\,,\,1/2)$ and define for $x = (x_1\,,x_2) \in \OQ$ 
$$
\nu_{x_2}(n) = \inf \{ k \in \N : x_2 + S_2(k) \ge \mf\,n^{1-3\ep/2}\}.
$$
(Recall that $\mf = \Ex(X_2)$.)

\begin{lem}\label{lem:nux} 
\hspace*{2.35cm} $\Prob[ \nu_{x_2}(n) > n^{1-\ep}] \le C \, n^{-(1-\ep)(1+\delta/2)}\,.$
\end{lem}

\begin{proof}
Let $\sigma_2^2 = \Var(X_2)$. Setting $p = 2+\delta$, we apply the first inequality in 
Lemma \ref{lem:Mn} to $X_2 - \mf$ and use Markov's 
inequality\footnote{Here and in many subsequent 
instances, values such as $n^{1-\ep}$ should be rounded to the next lower, resp. upper
integer. This is omitted in the notation, since it will be clear from the context.}:
$$
\begin{aligned}
 \Prob\bigl[ \nu_{x_2}(n) > n^{1-\ep}\bigr] 
& \le \Prob\bigl[x_2 + S_2(n^{1-\ep}) < \mf\,n^{1-3\ep/2}\bigr]\\
&\le \Prob\bigl[S_2(n^{1-\ep}) -\mf\,n^{1-\ep} < - \mf\,(1-n^{-\ep/2})\, n^{1-\ep}\bigr]\\
& \le \Prob\bigl[\,|S_2(n^{1-\ep}) -\mf\,n^{1-\ep}|^p > \bigl(\mf\,(1-n^{-\ep/2})\bigr)^p\, 
n^{(1-\ep)p}\bigr]\\
&\le \frac{ c_p\, n^{(1-\ep)p/2}\,\Ex\bigl(|X_2-\mf|^p\bigr)}
{\bigl(\mf\,(1-n^{-\ep/2})\bigr)^p\,n^{(1-\ep)p}}
\,.
\end{aligned}
$$
This proves the lemma. 
\end{proof}

\begin{proof}[Proof of Theorem \ref{thm:main2}(d)]
We choose $s$, $\ep$ and $\phi_n$ as follows.
\begin{equation}\label{eq:ep}
\ep = \frac12 - \frac{1}{s} \AND \phi_n = n^{-\ep/s} + n^{-\ep}\,,\quad s > 2\,.
\end{equation}
For the starting point $x \in \OQ$, we assume that $x_1 < n^{1/2-\ep} = n^{1/s}$ and
$x_2 < \mf\, n^{1-3\ep/2}$.  
$$
\begin{aligned}
\Prob[\tau_x > n] = \Prob[ \tau_x > n\,;\, \nu_{x_2}(n) \le n^{1-\ep}] &+
                     \Prob[\nu_{x_2}(n) \ge \tau_x > n] \\
                      &+ \Prob[\tau_x > n\,;\,\tau_x > \nu_{x_2}(n) > n^{1-\ep}]\\
=  \Prob[ \tau_x > n\,;\, \nu_{x_2}(n) \le n^{1-\ep}] &+ \mathfrak{o}(n^{-1/2}), 
\end{aligned}
$$
because by Lemma \ref{lem:nux},
$$
\Prob[\nu_{x_2}(n) \ge \tau_x > n] + \Prob[\tau_x > n\,;\,\tau_x > \nu_{x_2}(n) > n^{1-\ep}]
\le \Prob[ \nu_{x_2}(n) > n^{1-\ep}] \le C \, n^{-(1-\ep)(1+\delta/2)}\,.
$$ 
Now we write 
$$
\begin{aligned}
&\Prob[ \tau_x > n\,;\, \nu_{x_2}(n) \le n^{1-\ep}] \\
&\qquad = \sum_{k=1}^{n^{1-\ep}}\int_{y\,:\,y_2 \,\ge\, \mf\,n^{1-3\ep/2}} 
\Prob[\tau_x > k\,,\; \nu_{x_2}(n)=k\,,\; x+S(k)=dy] \,\, 
\Prob[\tau_y > n-k]\,.
\end{aligned}
$$
We note that
$\Prob[\tau_{y_1} > n-k] - \Prob[\tau_{y_2} \le n] 
\le \Prob[\tau_y > n-k] \le \Prob[\tau_{y_1} > n-k]$.
By \eqref{eq:Min}, we\\[3pt] have for $y_2 \ge \mf\,n^{1-3\ep/2}$ that
$$
\Prob[\tau_y > n-k] = \Prob[\tau_{y_1} > n-k] + \mathfrak{o}(n^{-(1-3\ep/2)(2+\delta)})
= \Prob[\tau_{y_1} > n-k] + \mathfrak{o}(n^{-1/2})
$$
by our choices of $\ep$ and $s$. 
Therefore, independently of the choice of $x$, 
$$
\begin{aligned}
&\Prob[ \tau_x > n\,;\, \nu_{x_2}(n) \le n^{1-\ep}] \\
&\quad = \sum_{k=1}^{n^{1-\ep}}\int_{y\,:\,y_2 \,\ge\,\mf\, n^{1-3\ep/2}} 
\Prob[\tau_x > k\,,\; \nu_{x_2}(n)=k\,,\; x+S(k)=dy] \,\, 
\Prob[\tau_{y_1} > n-k] \; + \mathfrak{o}(n^{-1/2})\\
&\quad = \circa + \circb + \mathfrak{o}(n^{-1/2})\,, \qquad\qquad \text{where}\\[6pt]
&\quad \circa = \sum_{k=1}^{n^{1-\ep}}
\int_{\scs y\,:\,y_2 \,\ge\,\mf\, n^{1-3\ep/2}  \atop \scs y_1 \,\le\, \phi_n\,n^{1/2}} 
\Prob[\tau_x > k\,,\; \nu_{x_2}(n)=k\,,\; x+S(k)=dy] \,\, 
\Prob[\tau_{y_1} > n-k] \AND\\
&\quad \circb = \sum_{k=1}^{n^{1-\ep}}
\int_{\scs y\,:\,y_2 \,\ge\,\mf\, n^{1-3\ep/2} \atop \scs y_1 \,>\, \phi_n\,n^{1/2}} 
\Prob[\tau_x > k\,,\; \nu_{x_2}(n)=k\,,\; x+S(k)=dy] \,\, 
\Prob[\tau_{y_1} > n-k] \,.
\end{aligned}
$$
By propositions \ref{pro:harmonic1}, \ref{pro:dim1} and \ref{pro:harmonic}, we have 
with $\kappa = (\pi\, \sigma_2^2/2)^{-1/2}$
$$
\Prob[\tau_{y_1} > n-k] 
\sim \kappa\, h_1(y_1) \,n^{-1/2}  \sim \kappa\, h(y) \,n^{-1/2} 
$$
uniformly for $y$ in the range of integration of term $\circa$ and $k \le n^{1-\ep}$.
Therefore
$$
\begin{aligned}
\circa &\sim \kappa\, n^{-1/2} \,
\Ex \biggl( h\Bigl(x+S\bigl(\nu_{x_2}(n)\bigr)\Bigr)\,;\, {\displaystyle
\tau_x > \nu_{x_2}(n)\,,\;\nu_{x_2}(n) \le n^{1-\ep}\,,\;  \atop \displaystyle
x_1 + S_1\bigl(\nu_{x_2}(n)\bigr) \le \phi_n\,n^{1/2} }\biggr)\\
&= \kappa\, n^{-1/2} ( \circc - \circd)\,,\qquad \text{where}\\[5pt]
\circc &= \Ex \Bigl( h\Bigl(x+S\bigl(\nu_{x_2}(n)\bigr)\Bigr)\,;\, 
\tau_x > \nu_{x_2}(n)\,,\;\nu_{x_2}(n) \le n^{1-\ep}\Bigr) \quad\AND\\
\circd &= \Ex \Bigl( h\Bigl(x+S\bigl(\nu_{x_2}(n)\bigr)\Bigr)\,;\, 
\tau_x > \nu_{x_2}(n)\,,\;\nu_{x_2}(n) \le n^{1-\ep}\,,\; 
x_1 + S_1\bigl(\nu_{x_2}(n)\bigr) > \phi_n\,n^{1/2}\Bigr).
\end{aligned}
$$
Also, $\Prob[\tau_{y_1} > n-k] \le C \, y_1 \,n^{-1/2}$ 
for  $y$ in the range of integration of term $\circb$. Therefore we get that both
$n^{1/2}\,\circb$ and $\circd$ are bounded above by $C\, \circd'$, where
$$
\begin{aligned}
\circd' &= \Ex \Bigl( x_1+S_1\bigl(\nu_{x_2}(n)\bigr)\,;\, 
\tau_x > \nu_{x_2}(n)\,,\;\nu_{x_2}(n) \le n^{1-\ep}\,,\; 
x_1 + S_1\bigl(\nu_{x_2}(n)\bigr) > \phi_n\,n^{1/2}\Bigr)\\
&\le  C \, \Ex \bigl( n^{1/2-\ep}+\Mmax_1(n^{1-\ep})\,;\, 
n^{1/2-\ep} + \Mmax_1(n^{1-\ep}) > \phi_n\,n^{1/2}\bigr)\,.
\end{aligned}
$$
We now apply Lemma \ref{lem:Mn} to the first marginal. 
We have $\phi_n\,n^{1/2} - n^{1/2-\ep} = n^{1/2 - \ep/s}\,.$
Next, we choose $p=2+\delta$, where $\Ex(|X|^{2+\delta}) < \infty$.
Furthermore, we substitute
$n^{1-\ep} = k= k_n\,$, whence 
$$
\begin{aligned}
n^{1/2-\ep/s} &= k^{1/2 + \alpha} \AND n^{1/2 - \ep} = k^{1/2 - \beta}\,,
\quad \text{where}\\ 
\alpha &= \frac{\ep}{1-\ep}\Bigl(\frac{1}{2}-\frac{1}{s}\Bigr)
\AND \beta = \frac{\ep}{2(1-\ep)}.
\end{aligned}
$$
Thus, using $k$ in the place of $n$ in Lemma \ref{lem:Mn},
$$
\begin{aligned}
\Ex \bigl( n^{1/2-\ep}&+\Mmax_1(n^{1-\ep})\,;\, 
n^{1/2-\ep} + \Mmax_1(n^{1-\ep}) > \phi_n\,n^{1/2}\bigr)\\
&= \Ex \bigl( k^{1/2-\beta}+\Mmax_1(k)\,;\, 
\Mmax_1(k) > k^{1/2 + \alpha}\bigr)
\le  c_p \, \Ex\bigl(|X|^p\bigr)\,\frac{2p-1}{p-1}\, k^{1/2 - (p-1)\alpha}\,.
\end{aligned}
$$
We need that $1/2 - (p-1)\alpha < 0$, where $p=2+\delta$. This can be achieved
by choosing $s$ sufficiently large (*). Then $\circd' \to 0$ uniformly for 
$x_1 \le n^{1/2-\ep} = n^{1/s}$, independently of~$x_2\,$. 

\medskip

We come to the estimation of the principal term $\circc$.
With $\gamma_{x_2}(n) = \min \{ \nu_{x_2}(n), n^{1-\ep}\}$,  
$$
\begin{aligned}
\circc &= \Ex \Bigl( h\Bigl(x+S\bigl(\gamma_{x_2}(n)\bigr)\Bigr)\,;\, 
\tau_x > \gamma_{x_2}(n)\,,\;\nu_{x_2}(n) \le n^{1-\ep}\Bigr) =  h(x) - \circe\,,\quad \text{where}\\[5pt]
\circe &= \Ex \Bigl( h\bigl(x+S(n^{1-\ep})\bigr)\,;\, 
\tau_x > n^{1-\ep}\,,\;\nu_{x_2}(n) > n^{1-\ep}\Bigr). 
\end{aligned}
$$
At last, once more since $h(x) \le C\, x_1$ by propositions 
\ref{pro:harmonic} and \ref{pro:harmonic1},
$$
\begin{aligned}
 \circe &\le C\,\Ex \bigl( x_1+S_1(n^{1-\ep}) \,;\, 
\tau_x > n^{1-\ep}\,,\;\nu_{x_2}(n) >  n^{1-\ep}\Bigr)\\
&\le C\, \Bigl(\Ex \bigl(x_1+S_1(n^{1-\ep})\bigr)^2\Bigr)^{1/2}\,
\Bigl(\Prob[\tau_x > n^{1-\ep}\,,\;\nu_{x_2}(n) > n^{1-\ep}]\Bigr)^{1/2}\\
&\le C \, \bigl(2\,x_1^2 + 2\,n^{1-\ep}\sigma_1^2\bigr)^{1/2}\,
\Bigl(\Prob[\nu_{x_2}(n) > n^{1-\ep}]\Bigr)^{1/2}, 
\end{aligned}
$$
which tends to $0$ as $n \to \infty$ by Lemma \ref{lem:nux} and our assumption that
$x_1 \le n^{1/2-\ep} = n^{1/s}$. Recall from (*) above
that $s$ depends on the $\delta$ of the moment condition for $X_1\,$.
\end{proof}

\section{The 2-dimensional Lindley process}\label{sec:Lindley}

We first explain how the queueing process is related with the exit times.

\begin{proof}[Proof of Lemma \ref{lem:relation}]
 In a good number of references, the increments in the definition
of $W(n)$ come with a ``plus'' sign. We have chosen the ``minus'' because
this is more convenient when relating the process with the exit times. In
{\sc Feller's} second volume~\cite[VI.9]{Fe2}, the ``plus'' is used, and the
one-dimensional case is considered. See in particular 
\cite[Theorem on p. 198]{Fe2}.  That theorem applies without changes 
also to higher dimensions,
and rewritten in terms of the ``minus'' sign, it says the following for 
$x = (x_1\,,x_2) \in \R_+^2\,$.
$$
\begin{aligned}
\Prob\bigl[W^0(n) \in [0\,,\,x_1) \times [0\,,\,x_2)\bigr] 
= \Prob \bigl[ - \Mmin_1(n) < x_1\,,\; - \Mmin_2(n) < x_2 \bigr]& \\ 
= \Prob \bigl[ x + S(k) \in \OQ \;\text{ for all } \; k \le n\bigr]& 
= \Prob[\tau_x > n]. 
\end{aligned}
$$
This concludes the proof.
\end{proof}

We now consider recurrence versus transience of the 2-dimensional  Lindley
process.

\bigskip

\textbf{Negative drift in at least one coordinate}

\smallskip

\begin{proof}[Proof of Theorem \ref{thm:main1}(a)] Suppose that 
$\Ex(X_1^+) < \Ex(X_1^-) \le \infty\,$. Then $S_1(n) \to -\infty$ almost surely.
The times when the first marginal process $W_1^0(n)$ starting at $0$ visits $0$ are the 
non-strictly increasing ladder epochs of $S_1(n)$, see \cite{Li} and 
\cite[VI.9]{Fe2} (and keep in mind the ``plus/minus'' sign issue mentioned above). 
It is well known that  the latter terminates almost surely, so that $W_1^0(n) \to \infty$ and 
thus also $W_1^{x_1}(n) \to \infty$ almost surely for every $x_1 \ge 0\,$: 
the first marginal process is transient,
whence also the 2-dimensional process is transient. 
\end{proof}

\bigskip

\textbf{Zero drift and negative correlation}

\smallskip

We now consider the case when both marginals have zero drift, but the 
correlation is negative.

\begin{proof}[Proof of Theorem \ref{thm:main1}(c), transient case]
Under the moment conditions of  Theorem \ref{thm:main1}(c), when 
$\rho(X_1\,,X_2) < 0$ then the combination of Lemma \ref{lem:relation} with
Theorem \ref{thm:main2}(c) shows that
$$
\sum_{n=0}^{\infty} \Prob[ W^0(n) \in [0\,,\,x_1) \times [0\,,\,x_2)] < \infty \quad
\text{for all }\; x_1\,,\,x_2 > 0. 
$$
This is the expected number of visits of the process to each of those rectangles.
Therefore, with probability 1, each rectangle is visited only finitely often, so that
$|W^0(n)| \to \infty\,$ and the process is transient.  
\end{proof}

\bigskip

\textbf{Recurrence}

\smallskip

We now turn our attention the the remaining cases, that is, statements (b) and (d)
of Theorem \ref{thm:main1}, as well as (c) with $\rho(X_1\,,X_2) \ge 0$.

\begin{proof}[Conclusion of the proof of Theorem \ref{thm:main1}]
In all of those remaining cases, we have by Theorem \ref{thm:main2}
$$
\sum_{n=0}^{\infty} \Prob\bigl[ W^0(n) \in [0\,,\,x_1) \times [0\,,\,x_2)\bigr] = \infty \quad
\text{for all }\; x_1\,,\,x_2 > 0. 
$$
In the case when the distribution of $X = (X_1\,,X_2)$ is non-lattice, 
this alone does not guarantee that the process is topologically recurrent.
But thanks to Assumption \ref{ass:basic}(ii), there is $x = (x_1\,,x_2) \in \OQ$
such that for each $n$,
$$
\Prob\bigl[X(n+1) \in [x_1\,,\,\infty) \times [x_2\,,\infty)\bigr] 
= \Prob\bigl[X \in [x_1\,,\,\infty) \times [x_2\,,\infty)\bigr]  > 0\,.
$$
But if $W^0(n) \in [0\,,\,x_1) \times [0\,,\,x_2)$ and 
$X(n+1) \in [x_1\,,\,\infty) \times [x_2\,,\infty)$ then $W^0(n+1)= 0$.
Therefore, for this choice of $(x_1\,,x_2)$, and since $W^0(n)$ and $X(n+1)$
are independent, 
$$
\begin{aligned}
 \sum_{n=0}^{\infty} \Prob[ W^0(n+1) = 0 ] &\ge 
\sum_{n=0}^{\infty} \Prob\bigl[ W^0(n) \in [0\,,\,x_1) \!\times\! [0\,,\,x_2)\,,\;
X(n+1) \in [x_1\,,\,\infty) \!\times\! [x_2\,,\infty)\bigr]\\
&=  \Prob\bigl[X \in [x_1\,,\,\infty) \!\times\! [x_2\,,\infty)\bigr]
\sum_{n=0}^{\infty} \Prob\bigl[ W^0(n) \in [0\,,\,x_1) \!\times\! [0\,,\,x_2)\bigr] = \infty. 
\end{aligned}
$$
Since $0$ is a single state of our Markov process, this implies via basic Markov chain 
theory that it is a recurrent state: 
$$
\Prob[ W^0(n)=0 \;\text{ infinitely often}\,]=1\,.
$$
Now let us look at this under the viewpoint of stochastic
dynamical systems (SDS): the random mappings $F_k(x) = \max \{ 0, x- X(k)\}$
are elements of the semigroup $\mathfrak{L}$ of contractions of $\R_+^2$ 
with Liptschitz constants $\le 1$. We have
$$
W^x(n) = F_n \circ \dots \circ F_1(x)\,.
$$
In particular, $|W^x(n) - W^y(n)|$ is decreasing in $n$, whence 
$|W^x(n) - W^0(n)| \le |x|$. Consequently, our SDS is non-transient:
$$
\Prob\bigl[ |W^x(n)| \to \infty \bigr] = 0 \quad \text{for every }\;
x \in \R_+^2\,.
$$
$\mathfrak{L}$ carries the topology of uniform convergence on compact
sets. Denote by $\wt \mu$ the (common) distribution of the 
i.i.d. random mappings $F_n\,$. Then the above arguments which led to recurrence of
the state $0$ entail
that the constant mapping $x \mapsto 0$ can be approximated in $\mathfrak{L}$
by a sequence $f_n \circ \dots \circ f_1\,$, $n \in \N\,$, where each function 
$f_k$ is in the support of $\wt \mu$. 
At this point, we can invoke a result going back to \cite{Le}, developped further
in \cite{Be} and \cite{PeWo1}; see {\sc Kloas and Woess}~\cite[Prop. 2.5]{KlWo} for
a compact formulation. It implies the existence of a limit set $\Ll$ and recurrence
as stated in the Introduction in \eqref{eq:limitset} and the preceding paragraph.
\end{proof}

\bigskip

\textbf{Discussion}

\smallskip

\emph{Invariant measures.} In the cases where the two-dimensional Lindley process is recurrent,
it follows from \cite{PeWo1} that it has a unique invariant measure $\nu$ up to constant factors.
It is supported on the limit set $\Ll$, and as a starting measure for the process, it makes
the time shift ergodic. For details, see \cite[Thm. 2.13]{PeWo1}.
In particular, its marginals $\nu_1$ and $\nu_2$ are the unique invariant measures for the
respective marginal processes. It is well understood that the invariant measure for 
a recurrent one-dimensional Lindley process has finite total mass when the increment has 
positive expectation. (Recall that we \emph{substract} the increment.) Also, that measure has
infinite total mass in the drift-free case. This is the reason why for the two-dimensional process,
$\nu$ has finite total mass (positive recurrence) 
in Theorem \ref{thm:main1}(b), while it has infinite total mass (null recurrence)
in Theorem \ref{thm:main1}(c) and (d).  

\medskip

\emph{The roles of Assumption \ref{ass:basic}(ii) and of the condition
$\rho(X_1\,,X_2) \ge 0$.} Consider the typical example of a clerk at a counter serving
a queue of customers. In that case, the waiting time of the $n^{\text{th}}$ customer
is modelled by a one-dimensional Lindley process, and recurrence means that the process
returns to state $0$ with probability one - the queue will be empty and the clerk 
at the counter will be able to have a break. In the two-dimensional case, we have two
counters and two queues, and in general we assume that the clerks do not operate
independently. Then both Assumption \ref{ass:basic}(ii) and the non-negative correlation
can be interpreted in the sense that the two clerks are co-operative, which will give
them the possibility to have a joint break.

Assumption \ref{ass:basic}(ii) was used in the proof of Proposition \ref{pro:harmonic}
in order to show that $h > 0$ strictly on all of $\OQ$. Without it, our proof does not
allow us to go beyond the statement that
for every $b_1 > 0$ (possibly small) there is 
$b_2 > 0$ (possibly large) such that $h > 0$ on
$(b_1\,,\,\infty) \times (b_2\,,\,\infty)$.

In the proof of Theorem \ref{thm:main1}(d) this would lead to the restriction that
we only get
$$
\sum_{n=0}^{\infty} \Prob\bigl[ W^0(n) \in [0\,,\,x_1) \times [0\,,\,x_2)\bigr] = \infty 
$$
for sufficiently large $x_1$ and $x_2$. 

Now, in the case when the process is discrete (in which case we assume without loss
of generality that the increments are in $\Z^2$ and the process evolves within
$\N_0^2$), this does imply recurrence in the sense that the process visits each point
of the (integer) limit set infinitely often with probability one. However, the origin is
then not necessarily part of the limit set: the two clerks will not be able to have
a joint break.

On the other hand, when the process is non-discrete then it is by no means clear that divergence
of the above series implies topological recurrence -- here, recurrence was deduced because
on the basis of Assumption \ref{ass:basic}(ii) we could show that
$\sum_n \Prob\bigl[ W^0(n) = 0] = \infty$.     

\medskip

\emph{The limit set.} $\Ll$ depends on the support of the distribution of the increments.
In our case, most importantly, it contains the origin. 
It appears quite hard to determine the full limit set 
explicitly, compare with \cite{KlWo}. 
\medskip

\emph{Moment assumptions.}
In theorems \ref{thm:main1}(d) and \ref{thm:main2}(d), we have assumed among other
that $\Ex((X_2^-)^{3+\delta}) < \infty\,$. This is used in the proof of Claim 1 in Proposition
\ref{pro:harmonic}. We believe that a finite moment of order
$2 + \delta$ should suffice, but so far we have not succeeded in proving this. 

\medskip

\emph{Higher dimensions.} We can consider the same issues on $\R_+^d$ for any $d \ge 2$.
The model random variable for the increments is then $X = (X_1\,,\dots, X_d)$, and the definition
of $W(n)$ remains the same, as well as Assumption \ref{ass:basic} with the obvious adaptation. 

If there is one coordinate with $\Ex(X_i^+) < \Ex(X_i^-) \le \infty\,$
then $W(n)$ is transient, while if all coordinates satisfy 
$\Ex(X_i^-) < \Ex(X_i^+) < \infty\,$, then $W(n)$ is positive recurrent.

Also, if one coordinate is centered and all others have positive expectiation, then
under the same moment conditions as in Theorem \ref{thm:main1}(d) one gets null recurrence;
the proof remains practically the same.

For recurrence, one cannot have more than two centered coordinates. So the main case still
to be studied is the analogue of Theorem \ref{thm:main1}(d) in the case when there are two
centered coordinates and all the other coordinates have positive expectation. 
There are several substantial additional issues to be tackled in this situation, which
we reserve for future work.

\bigskip

\textbf{Remarks on reflected random walk}

\smallskip

A model which is very similar to the queueing process is reflected random walk.
In dimension 1, with starting point $x > 0$, both processes evolve like a 
random walk $x- S(n)$ as long as it stays non-negative. When the walk enters the 
negative half-axis, the value is reset to zero for the queueing process $W^x(n)$, while 
for  reflected random walk $R(n) = R^x(n)$, the sign is changed. In arbitrary dimension, this
means that reflected random walk is given by
$$
R(0) = x \in \R^+\,, \quad R(n) = |R(n-1) - X(n)|\,,
$$
where (attention!) the absolute value is taken coordinate-wise.
The one-dimensional model is very well studied; instead of re-displaying all references,
we refer to \cite{PeWo1}, as well as to \cite{Pe} and \cite{KlWo}, where the recurrence 
of the multidimensional 
reflected random walk is  studied in the case where $\Ex(X_i) > 0$ for all coordinates. 

In general, it is easy to see that transience of the multi-dimensional queueing process
implies transience of the corresponding reflected random walk. In other words, the recurrence
issue is somewhat harder for reflected random walk. In the following case, the two processes
are recurrent, resp. transient simultaneously.

\begin{lem}\label{lem:simult}
If there is $\Delta > 0$ such that $X_i \ge -\Delta$ almost surely for $i=1,\dots, d$
then for the two processes starting at the same point, 
$$
R_i(n) \le X_i(n) + \Delta \quad \text{almost surely for $i=1,\dots, d$ and all $n$.}
$$
\end{lem}

The proof is straightforward. Thus, when the multidimensional Lindley process
visits the origin infinitely often with probability $1$, then the corresponding 
multidimensional reflected random walk visits the set $[0\,,\,\Delta]^d$
infinitely often with probability $1$. By the contractivity properties of
$\bigl(R^x(n)\bigr)$, this implies that the reflected random walk is topologically
recurrent on its limit set, compare with \cite{PeWo1} and \cite{KlWo}.

In dimension 2, this applies, in particular, to the situation of Theorem \ref{thm:main1}.

\section{Appendix: a proof due to D. Denisov and V. Wachtel}\label{sec:app}

In this subsection, the proof of Proposition \ref{pro:harmonic1} is completed. That is,
we present the elaboration of an argument of Denisov and Wachtel which proves
\begin{equation}
\lim_{x \to \infty} \frac{h_1(x)}{x} =1. \tag{\ref{eq:h1}}
\end{equation}

Given the centered real random variable $X$ with $\Ex(X^2) < \infty\,$,
we consider the following functions $a,  b, m: \R_+ \to \R_+$ associated with
the negative part $X^-$ of~$X$.
\begin{align*}
a(x) &= -\Ex(x+X\,;\,x+X\leq 0) = \int_x^{\infty} \Prob[X \le -y]\,dy\,, \\
b(x) &=  \int_x^{\infty} a(y)\,dy\,,\quad\text{and}  \quad
m(x) = \int_0^x b(y)\,dy\,,\quad x \ge 0.
\end{align*}
Similarly, we set $\displaystyle \bar a(x) = \int_x^{\infty} \Prob[X > y]\,dy.$ 
One easily verifies that $\;\displaystyle 2b(x) = \Ex\bigl((\max\{X^- -x, 0\})^{2}\bigr) = 
 \int_{-\infty}^x (x+y)^2dF(y) \to 0$ as $x \to \infty\,$,
whence also
\begin{equation}\label{eq:m}
 \lim_{x \to \infty} \frac{m(x)}{x} =0.
\end{equation}

\begin{pro}\label{pro:AR}
There are positive constants $R$ (sufficiently large) and $A$
such that the function
$$
V(x) = x + A\, m(x) + R\,,\quad x \ge 0\,,
$$ 
is superharmonic for the one-dimensional random walk $S(n)$ killed when exiting $\R_+\,$, that is,
$$
\Ex \Bigl( V\bigl(x + S(1)\bigr)\,;\, \tau_x > 1 \Bigr) \le V(x) \quad \text{for all }\; x \ge 0. 
$$
\end{pro}

Let us first show how this implies the desired result.

\begin{proof}[Proof of \eqref{eq:h1}]
Proposition \ref{pro:AR} implies that for all $x \ge 0$ and $n \in \N$
$$
\begin{aligned}
V(x) &\ge  \Ex \Bigl( V\bigl(x + S(n)\bigr)\,;\, \tau_x > n \Bigr)\\
&\ge \Ex \bigl( x + S(n)\,;\, \tau_x > n \bigr)\\
&= \Ex \bigl( x + S(\min\{n, \tau_x\})\,;\, \tau_x > n \bigr)\\
&= x - \Ex \bigl( x + S(\tau_x)\,;\, \tau_x \le n \bigr).
\end{aligned}
$$
The last identity holds by the martingale property, since $S(n)$
is a centered random walk. We infer that
$$
-\Ex \bigl( x + S(\tau_x)\,;\, \tau_x \le n \bigr) \le V(x) - x = A\, m(x) + R
$$
We can let $n \to \infty\,$, and by monotone convergence
$$
0 \le - \Ex \bigl( x + S(\tau_x)\,;\, \tau_x < \infty \bigr) \le A\, m(x) + R
$$
By \eqref{eq:m}, 
$$
\lim_{x \to \infty} \frac{1}{x}\, \Ex \bigl( x + S(\tau_x)\,;\, \tau_x < \infty \bigr) = 0\,.
$$
Since $h_1(x) = x - \Ex \bigl( x + S(\tau_x)\,;\, \tau_x < \infty \bigr)$, the result
follows.
\end{proof}

 \begin{proof}[Proof of Proposition \ref{pro:AR}]
We want to show that
$$
\Delta(x) = \Ex \Bigl( V\bigl(x + S(1)\bigr)\,;\, \tau_x > 1 \Bigr) - V(x) \le 0 \quad 
\text{for }\;x \ge 0.
$$
We write $F(x) = \Prob[X \le x]$ and $\overline F(x) = 1 - F(x)$. Using that $x= \Ex(x+X)$, 
$$
\begin{aligned}
\Delta(x) &= \Ex\bigl(x+X + A\, m(x+X) + R \,;\, X > -x\bigr) - x - A\,m(x) - R\\
 &= - \Ex ( x+X\,;\,X \le -x \bigr) - R\, F(-x) - A\,m(x)\,F(-x)\\
&\hspace*{4.45cm}
+ A \,\, \Ex\bigl(m(x+X)-m(x)\,;\, X > -x\bigr)\\
&= a(x) - R\, F(-x) - A\,m(x)\,F(-x)
+ A \int_{-x}^{\infty} \bigl(m(x+y) - m(x)\bigr)\, dF(y).
\end{aligned}
$$
We decompose the last integral and integrate twice by parts, recalling that
$m'(x) = b(x), b'(x) = -a(x)$, $a'(x)=-F(-x)$ and $\bar a'(x)=-\overline F(x)\,$ :
$$
\begin{aligned}
 \int_{-x}^{\infty} &\bigl(m(x+y) - m(x)\bigr)\, dF(y)\\
&= \int_{-x}^{0} \bigl(m(x+y) - m(x)\bigr)\, dF(y)
- \int_{0}^{\infty} \bigl(m(x+y) - m(x)\bigr)\, d\overline F(y)\\
&= m(x)\, F(-x) - \int_{-x}^{0} b(x+y)F(y)\,dy + 
\int_{0}^{\infty} b(x+y)\overline F(y)\, dy\\
&= m(x)\, F(-x)  + \int_{0}^{x} b(x-y)a'(y)\,dy
-\int_{0}^{\infty} b(x+y)\overline a'(y)\, dy\\
&= m(x)\, F(-x)  + b(0)a(x) - b(x)a(0) - \int_{0}^{x} a(x-y)a(y)\, dy\\
&\hspace*{4.8cm}+ b(x)\bar a(0) -  \int_{0}^{\infty} a(x+y)\bar a(y)\, dy\,.
\end{aligned}
$$
We observe that $b(x)\bar a(0) - b(x)a(0) = b(x)\bigl(\Ex(X^+) - \Ex(X^-)\bigr)=0$,
and recall that $b(0) = \Ex\bigl((X^-)^{2}\bigr)\big/2$. 
Combining these computations,
$$
\begin{aligned}
\Delta(x) &= a(x) - R\, F(-x) +A \,\frac{\Ex\bigl((X^-)^{2}\bigr)}{2} \, a(x) 
- A \! \int_{0}^{x} \!\!\!\! a(x-y)a(y)\, dy - A \! \int_{0}^{\infty} \!\!\!\! a(x+y)\bar a(y)\, dy\\
&\le a(x) - R\, F(-x) +A \,\frac{\Ex\bigl((X^-)^{2}\bigr)}{2} \, a(x) 
- A \! \int_{0}^{x} \!\!\!\! a(x-y)a(y)\, dy\,.
\end{aligned}
$$
Using that $a(x)$ is monotone decreasing, 
$$
\begin{aligned}
\int_{0}^{x} \!\! a(x-y)a(y)\, dy &= 2 \int_{0}^{x/2} \!\! a(x-y)a(y)\, dy\\ 
&\ge 2a(x)\bigl(b(0)-b(x/2)\bigr) = a(x)\Ex\bigl((X^-)^{2}\bigr) -2a(x)b(x/2).
\end{aligned}
$$
We now choose $A = 4\big/\Ex\bigl((X^-)^{2}\bigr)$. Then we get 
$$
\Delta(x) \le - R\, F(-x) + a(x)\bigl(2A\, b(x/2) -1\bigr) \le -R\, F(-x) + 3a(x).
$$
Since $b(x) \to 0$ as $x \to \infty$, there is $x_0 > 0$ such that
$2A\, b(x_0/2) -1 = 0$ and hence $\Delta(x) \le 0$ for all $x  \ge x_0\,.$

Now suppose first that $F(-x_0) > 0$. Then we choose $R = 3\Ex(X^-)/F(-x_0)$,
and for $x \le x_0$, we have  $\Delta(x) \le -R\, F(-x_0) + 3a(0) =0$.

Finally suppose that $F(-x_0) = 0$, that is, $\Prob[X > -x_0]=1$. Then also $a(x_0)=0$.
This time, we choose $R= 3x_0\,$. For $0 \le x < x_0$, there is $\xi \in (x\,,\,x_0)$
such that $a(x) = a(x_0) - (x_0-x)a'(\xi) = (x_0-x)F(-\xi)$. Then
$$
\Delta(x) \le 3a(x) - R\,F(-x) = 3(x_0-x)\bigl(F(-\xi) - F(-x)\bigr) -3xF(-x) \le 0.
$$
This concludes the proof.
\end{proof}


\begin{thebibliography}{22}

 \bibitem{As} Asmussen, S.: \emph{Applied Probability and Queues,}
2nd edition. Applications of Mathematics {\bf 51}, Springer-Verlag, New York, 2003.

\bibitem{BBE} Babillot, M., Bougerol, Ph., and Elie, L.: \emph{The 
random difference equation $X\sb n=A\sb nX\sb {n-1}+B\sb n$ in the 
critical case.}  Ann. Probab. {\bf 25}  (1997) 478--493.

\bibitem{Be} Benda, M.: \emph{Schwach kontraktive dynamische Systeme.} 
Ph. D. Thesis, Ludwig-Maximilans-Universit\"at M\"unchen (1998).

\bibitem{Be1} Benda, M.: \emph{Contractive stochastic dynamical systems.} 
Unpublished preprint, Ludwig-Maximilans-Universit\"at M\"unchen (1999).

\bibitem{Be2} Benda, M.: \emph{A reflected random walk on the half line.} 
Unpublished preprint, Ludwig-Maximilans-Universit\"at M\"unchen (1999).

\bibitem{Bo} Borovkov, A.A.: \emph{Stochastic Processes in Queueing Theory.} 
Springer-Verlag, New York-Berlin, 1976.

\bibitem{Bu} Burkholder, D. L.: \emph{Distribution function inequalities for martingales.} 
Ann. Probab. {\bf 1} (1973) 19--42. 

\bibitem{CyKl} Cygan, W., and Kloas, J.: \emph{On recurrence of the multidimensional 
Lindley process.} Electron. Commun. Probab. {\bf 23} (2018), no. 4, 1--14.

\bibitem{DeWa} Denisov, D., and Wachtel, V.: \emph{Random walks in cones.} 
Ann. Probab. {\bf 43} (2015) 992--1044.

\bibitem{Do} Doney, R.A: \emph{Local behaviour of first passage probabilities.} 
 Probab. Theory Related Fields {\bf  152} (2012) 559--588. 

\bibitem{Du} Duraj, J.: \emph{Random walks in cones: the case of nonzero drift.} 
Stochastic Process. Appl. {\bf 124} (2014) 1503--1518.

\bibitem{Fe2} Feller, W.: \emph{An Introduction to Probability Theory
and its Applications.} Volume II, 2nd edition, Wiley, New York (1971).

\bibitem{FM} Flatto L., and McKean, H.P.: \emph{Two queues in parallel.} 
Comm. Pure Appl. Math. {\bf 30} (1977) 255--263. 

\bibitem{GaRa} Garbit, R., and Raschel, K.: \emph{On the exit time from a cone
for random walks with drift.} Rev. Mat. Iberoam. {\bf 32} (2016) 511--532.

\bibitem{GrSh} Greenwood, P., and Shaked, M.: \emph{Fluctuations of Random Walk in $\R^d$ 
and Storage Systems.} Adv. Applied Probab. {\bf 9} (1977), 566--587.

\bibitem{Gu} Gut, A.: \emph{Stopped Random Walks. Limit theorems and applications.} 
2nd edition,  Springer, New York, 2009. 

\bibitem{IgLo} Ignatiouk-Robert, I., and Loree, Ch.: \emph{Martin boundary of a 
killed random walk on a quadrant.} Ann. Probab. {\bf 38} (2010) 1106--1142. 

\bibitem{Ja} Janson, S.: \emph{Moments for first-passage and last-exit times, the minimum, 
and related quantities for random walks with positive drift.} Adv. Appl. Prob. {\bf 18} (1986)
865--879.

\bibitem{Ke} Kendall, D. G.: \emph{Some problems in the theory of queues.}
J. Roy. Statist. Soc. Ser. B. {\bf 13}
(1951) 151--173; discussion: 173--185.

\bibitem{KiWo} Kiefer, J., and Wolfowitz, J.: \emph{On the theory of queues with many servers.} 
Trans. Amer. Math. Soc. {\bf 78} (1955) 1--18.

\bibitem{Ki} Kingman, J. F. C.: \emph{Two similar queues in parallel.} Ann. Math. Statist.
{\bf 32} (1961),1314--1323. 

\bibitem{KlWo} Kloas, J., and Woess, W.: \emph{Multidimensional random walk with reflections}. 
Stochastic Proc. Appl. {\bf 129} (2019) 336--354.

\bibitem{Le} Leguesdron, P.: \emph{Marche al\'eatoire sur le semi-groupe des contractions 
de $\R^d$. Cas de la marche al\'eatoire sur $\R_+$ avec choc \'elastique en z\'ero.} 
Ann. Inst. H. Poincar\'e Probab. Statist. {\bf 25} (1989) 483--502.

\bibitem{Li} Lindley, D. V.: \emph{The theory of queues with a single server.} 
Proc. Cambridge Philos Soc. {\bf 48}(1952), 277--289.

\bibitem{Pe} Peign\'e, M.: \emph{Marches de Markov sur le semi-groupe
des contractions de $\R^d$. Cas de la marche al\'eatoire \`a pas markoviens
sur $(\R_+)^d$ avec chocs \'elastiques sur les axes.}  Ann. Inst. H. 
Poincar\'e Probab. Stat. {\bf 28} (1992) 63--94.

\bibitem{PeWo1} Peign\'e, M., and Woess, W.: \emph{Stochastic dynamical systems 
with weak contractivity properties, I. Strong and local contractivity.} Colloq. Math. 
{\bf 125} (2011) 31–54. 

\bibitem{PeWo2} Peign\'e, M., and Woess, W.: \emph{Stochastic dynamical systems with 
weak contractivity properties, II. Iteration of Lipschitz mappings.} Colloq. Math. 
{\bf 125} (2011) 55--81. 

\bibitem{Ph} Pham, T. D. C.: \emph{Conditioned limit theorems for products of 
positive random matrices.}
ALEA Lat. Am. J. Probab. Math. Stat. {\bf 15} (2018) 67--100. 



\end{thebibliography}
\end{document}